\documentclass[dvipdfmx,12pt]{article}

\usepackage{ascmac}
\usepackage{cite}
\usepackage{float}
\usepackage{wrapfig}
\usepackage{amsfonts}
\usepackage{amsmath}
\usepackage{amsbsy}
\usepackage{amssymb}
\usepackage{amsthm}
\usepackage{bm}
\usepackage{latexsym}
\usepackage{color}
\usepackage{here}
\usepackage{here}
\usepackage[]{graphicx}
\usepackage[hang,small,bf]{caption}
\usepackage[subrefformat=parens]{subcaption}
\usepackage{cancel}

\numberwithin{equation}{section}

\theoremstyle{plain}
\newtheorem{thm}{Theorem}[section]
\newtheorem{lem}{Lemma}[section]
\newtheorem{cor}{Corollary}[section]
\newtheorem{prop}{Propositon}[section]
\theoremstyle{definition}

\theoremstyle{remark}
\newtheorem{rem}{Remark}

\begin{document}
	\rightline{\baselineskip16pt\rm\vbox to20pt{
			{
				\hbox{OCU-PHYS-516}
				\hbox{AP-GR-159}
			}
			\vss}}%

	\begin{center}
		{\LARGE\bf Characterization of 3D Sasakian manifold from magnetic Hopf surfaces}\\
		\bigskip\bigskip
		{\large 
			Satsuki Matsuno\footnote{smatsuno@sci.osaka-cu.ac.jp}
		}\\
		\bigskip
		{\it Department of Physics, Graduate School of Science,
			Osaka City University\\
			3-3-138 Sugimoto, Sumiyoshi, Osaka 558-8585, Japan}
	\end{center}
	
%
%
	\begin{abstract}
		In a three-dimensional Riemannian manifold M that admits a unit Killing vector field $\xi$, we regard $\xi$ as a magnetic vector field.
		A magnetic Hopf surface is a surface obtained by Lie dragging the magnetic curve with $\xi$. 
		Then we characterize Sasakian structure on M from magnetic Hopf surfaces.
		That is, we show that if an arbitrary magnetic Hopf surface is a constant mean curvature surface then M is a Sasakian manifold.
	\end{abstract}

	\section{Introduction}
	If one observes trajectories of charged particles in a magnetic field in a Riemannian manifold, one can obtain information of both the magnetic field and geometry of the manifold.
	In this paper we discuss such an attempt for a static magnetic field in three dimensions in simple mathematical settings.
	
	A magnetic field in a Riemannian manifold $(M,g)$ is defined as a closed 2-form $F$.
	By the use of $F$, the $(1,1)-$type tensor field $\phi$ is defined by $g(X,\phi Y)=F(X,Y)$, and the Lorentz force that works on a particle of velocity $v$ and charge $q$ is given by $q\phi(v)$.
	An affine-parametrized trajectory $c(t)$ of the charged particle satisfies the differential equation $\nabla_{\dot{c}}\dot{c}=q\phi(\dot{c})$, which is called a magnetic curve for the magnetic field $F$.
	
	On a given manifold and a given magnetic field, a magnetic curve depends on both the geometry and the magnetic field.
	However, if we consider the magnetic field that is tied uniquely to the geometry of a manifold, a magnetic curve is essentially governed by the geometry.
	The magnetic curves for such a magnetic field are useful probes for investigating the geometry of manifold.
	There are well known examples for such cases: 
	K\"{a}hlerian magnetic field on a K\"{a}hlerian manifold, where we regard the K\"{a}hler form as a magnetic field and Sasakian magnetic field on a Sasakian manifold, where we also define the magnetic field by using the Sasakian structure.
	The magnetic curves were studied in the K\"{a}hlerian magnetic field \cite{Comtet,Sunada,Adachi 1995}, and in the Sasakian magnetic field\cite{Cabrerizo et al 2009,Romaniuc et al 2015}.
	
	An importance of K\"{a}hlerian and Sasakian manifolds is recognized in physics.
	K\"{a}hlerian manifolds play an important role of Super symmetric field theory and compactifications of extradimension.
	And five-dimensional Sasakian-Einstein manifolds are important in context of AdS/CFT correspondence\cite{Maldacena:1997re,Klebanov:1998hh}.
	In recent years, the significance of the Sasakian manifold has been revisited from the mathematical side.
	
	In general, we might not be able to uniquely determine both a manifold and a magnetic field by information of magnetic curves.
	However, we can characterize a Sasakian manifold and a Sasakian magnetic field by magnetic curves since the geometry of manifold and the magnetic field are essentially identical.
	Infact, it is shown that if any magnetic curves in the contact magnetic field of a three-dimensional contact metric manifold are slant helices then the manifold is a Sasakian manifold(\cite{Cabrerizo et al 2009} Theorem 5.3).
	The aim of this paper is giving another characterization of the Sasakian manifold.
	
	We consider a three-dimensional Riemannian manifold $(M,g)$ admits a unit Killing vector field $\xi$, and regard $\xi$ as the magnetic vector field.
	This magnetic vector field is a generalization of a translation invariant magnetic vector field in the three-dimensional Euclidian space.
	In this setting, we investigate the condition for magnetic curves that specify the Sasakian manifold. 
	
	On three-dimensional manifolds, we can simplify the problem since we can identify a magnetic 2-form field and a magnetic vector field.
	Moreover, we use the Killing submersion as a tool effectively.
	The Riemannian submersion $\pi:M\rightarrow B$ whose fibers are integral curves of a unit Killing vector field is called a Killing submersion\cite{Espinar Oliveira 2013,Manzano 2014}．
	The remarkable geometric property of the three-dimensional Riemannian manifolds that admit a unit Killing vector field $\xi$ is that the sectional curvature of arbitrary planes containing $\xi$ depend on only an integral curve of $\xi$ \cite{Espinar Oliveira 2013,Manzano 2014}．
	This fact is useful in our study since existence of a unit Killing vector field $\xi$ and constancy of sectional curvature of planes containing $\xi$ are equivalent to being a Sasakian manifold in three-dimensional Riemannian manifolds.
	
	For characterization, the geometric object that we focus on is the magnetic surface that is an alternative of magnetic tube.
	A magnetic tube is constructed by an arbitrary closed curve and magnetic field lines.
	The magnetic surface is a certain generalization of magnetic tube.
	To define magnetic surface, we use magnetic curve instead of closed curve, i.e. the surface which is constructed by a magnetic curve and magnetic field lines.
	A magnetic surface can be a "tube".
	As a main result of this paper, we show that if an arbitrary magnetic surface has constant mean curvature then the manifold has a Sasakian structure.
	
	The organization of this paper is following.
	In Section.\ref{section Preliminary}, we sammarize some definitions and facts in magnetic field, magnetic curve, Sasakian manifold, Riemannian submersion and Killing submersion.
	In Section.\ref{section Characterization of Sasaki structure from view point of magnetic Hopf surfaces}, we prove our result.
	\section{Preliminary}\label{section Preliminary}
	\subsection{Magnetic Fields and Magnetic Curves}
	In a Riemannian manifold $(M,g)$, a magnetic field $F$ is a closed 2-form. 
	There locally exists a 1-form $A$, called an electromagnetic potential, such that $F=dA$.
	The tensor field $\phi$ of $(1,1)-$type is defined by $g(X,\phi Y)=F(X,Y)$ for arbitrary vector fields $X,Y$.
	Then $\phi$ gives the Lorentz force in the physical terminology.
	
	An affine-parametrized trajectory $c(t)$ of charge $q$ is described by the equation
	\begin{align}
	\nabla_{\dot{c}}\dot{c}=q\phi(\dot{c}).
	\end{align}
	This curve is called a magnetic curve for the magnetic field $F$.
	The speed of a magnetic curve $||\dot{c}(t)||$ is constant.
	A magnetic curve of unit speed and unit charge is called a normal magnetic curve.
	
	Since a magnetic curve is a generalization of geodesic, one can consider the Jacobi field and completeness for magnetic curves as same as the geodesics\cite{Adachi 1997}．
	Let $\gamma_{v_0}(t)$ be a magnetic curve of initial velocity $v_0$ and unit charge for a magnetic field $F$. 
	We should note that $\gamma_{v_0}(t)=\gamma_{av_0}(t/a)$ dose not hold for $a(\ne1)\in\mathbb{R}$.
	This is because physically the magnitude of Lorentz force depends on the speed of a charged particle.
	If $\gamma_{v_0}(t)$ is a magnetic curve of a magnetic field $F$ then $\gamma_{av_0}(t/a)$ is the magnetic curve for the magnetic field $aF$ or the magnetic curve of charge $1/a$ for the magnetic field $F$.
	
	Considering this circumstances, for fixed $F$ the magnetic exponential map $\mathbb{F}{\rm exp}$ is defined as follows.
	For an arbitrary point $p\in M$ the magnetic exponential map from a neighborhood V of $T_pM$ containing $o_p$ to a neighborhood U of $p$, $\mathbb{F}{\rm exp}_p:V\rightarrow U$, is defined by
	\begin{align}
	\mathbb{F}{\rm exp}_pv:=\begin{cases}
	\gamma_{v_0}(||v||)\ \ {\rm if}\ v\ne0\\
	p\ \ \ \ \ \ \ \ \ \ \ {\rm if}\ v=0
	\end{cases}
	\end{align}
	where $v_0=v/||v||$.
	
	The differential map of $\mathbb{F}{\rm exp}_p$ on $o_p\in T_pM$, $D\mathbb{F}{\rm exp}_po_p:T_{o_p}T_pM\simeq T_pM\rightarrow T_pM$, is identity. 
	Indeed,
	\begin{align}
	D\mathbb{F}{\rm exp}_po_p(v)=\frac{d}{dt}\gamma_{v_0}(t||v||)|_{t=0}=||v||v_0=v,
	\end{align}
	holds.
	
	Therefore the next lemma is obtained by inverse function theorem.
	\begin{lem}
		[\cite{Adachi 1997} Lemma 1.4 (3)]
		
		We can find a neighborhood $U$ of $p\in M$ and a positive number $\varepsilon$ such that $\mathbb{F}{\rm exp}_p|_{B_\varepsilon(o_p)}$ is a diffeomorphism onto some open set containing $U$, where $B_\varepsilon(o_p)$ is an $\varepsilon-$open ball.
	\end{lem}
	This lemma guarantees the following.
	\begin{lem}
		\label{lem local completeness of magnetic exponential map}
		Let $(M,g)$ be a complete Riemannian manifold and let $F$ be a magnetic field 2-form.
		Then two arbitrary points $p,q$ are connected by a piecewise smooth normal magnetic curve for the magnetic field F.
	\end{lem}
	In particular, for a two dimensional Riemannian manifold, we have
	\begin{lem}
		\label{lem local completeness of circle on 2dim}
		Let $(M,g)$ be an orientable complete two-dimensional Riemannian manifold and let $\Omega$ is a volume form.
		Then two arbitrary points $p,q\in M$ are connected by a piecewise smooth curve $\alpha(t):I=[a,b]\rightarrow M$ such that there is a partition $a=t_0<t_1<\cdots<t_{m-1}<t_m=b$ so that $\alpha(t)|_{[t_{k-1},t_k]}$ is the smooth curve satisfies
		\begin{align}
		\nabla_{\dot{\alpha}}\dot{\alpha}=\kappa n,\label{eq circle 1}\\
		\lim_{t\rightarrow t_{k-1}+0}\dot{\alpha}(t)=v_{k-1}^+,\ \lim_{t\rightarrow t_k-0}\dot{\alpha}(t)=v_k^-\label{eq circle 2}\\
		|v_{k-1}^+|=|v_k^-|=a={\rm const.}\label{eq circle 3}
		\end{align}
		for all $k\in\{1,\cdots,m \}$.
		Here $n$ is the unit vector field on $\alpha$ defined by $g(\dot{\alpha},n)=0,\ \Omega(\dot{\alpha},n)>0$ and $\kappa\ne0$ is a fixed constant.
		We call the curves satisfying the equation \eqref{eq circle 1} circles.
	\end{lem}
	\begin{proof}
		Define the $(1,1)-$type tensor field $J$ by $g(JX,Y)=\Omega(X,Y)$, then $(M,J,g)$ is a K\"{a}hlerian manifold of K\"{a}hler form $\Omega$.
		For a circle that satisfies equations \eqref{eq circle 1},\eqref{eq circle 2},\eqref{eq circle 3}, $J(\dot{\alpha})$ is in proportion to $n$ since $g(J\dot{\alpha},\dot{\alpha})=0$.
		Moreover, since we have $g(J(\dot{\alpha}),J(\dot{\alpha}))=g(\dot{\alpha},\dot{\alpha})=a$, and $g(J\dot{\alpha},n)=\Omega(\dot{\alpha},n)>0$, thus $J\dot{\alpha}=an$ holds.
		Then the curve $\alpha(t)$ is a piecewise $C^2$ normal magnetic curve for the K\"{a}hler magnetic field $F:=\kappa/a^3\Omega$.
		Therefore we obtain our claim by lemma.\ref{lem local completeness of magnetic exponential map}.
	\end{proof}
	\begin{rem}
		In general, the word "circle" means an arc-length-parametrized curve satisfying the equation \eqref{eq circle 1}\cite{Nomizu Yano 1974}．
	\end{rem}
	
	\subsection{Killing Magnetic Field and Sasaki Magnetic Field}
	Let $(M,g)$ be a three-dimensional orientable Riemannian manifold and let $B$ be a divergence free vector field.
	For a volume form $\Omega$, the 2-form $F_B:=\iota_B\Omega$ defines a magnetic field on $M$ since $dF_B=divB\Omega=0$.
	This defines a one-on-one correspondence between closed 2-form fields and divergence free vector fields.
	Therefore one can identify magnetic 2-form fields and magnetic vector fields.
	We would also say that $B$ is the magnetic vector field of the magnetic field $F_B$.
	Moreover the Lorentz force of $F_B$ is given by $\phi(X)=X\wedge B$, where $\wedge$ indicates vector product defined by $g(X\wedge Y,Z)=\Omega(X,Y,Z)$ for all vector field $X,Y,Z$.
	Especially, for a Killing vector field $\xi$, $F_\xi$ is called a Killing magnetic field\cite{Cabrerizo et al 2009}.
	
	Let $(M,\eta,\phi,\xi,g)$ be a Sasakian manifold whose contact form, endmorphism field, Reeb vector field and Riemannian metric compatible with the endmorphism field are $\eta,\ \phi,\ \xi$ and $g$ respectively.
	Sasakian magnetic field is defined by the closed 2-form $F:=d\eta$.
	The endmorphism field $\phi$ gives the Lorentz force since $g(X,\phi Y)=d\eta(X,Y)$.
	
	In particular, in the case of $\dim M=3$, $\phi(X)=X\wedge\xi$ holds for all vector field $X$, i.e. the Sasakian magnetic field is a Killing magnetic field described by Reeb vector field $\xi$(See Appendix \ref{appen sasaki manifold}. Proposition \ref{appen prop sasaki and killing magnetic field}).
	And we should note that we can regard the Reeb vector field $\xi$ as a magnetic vector field only in three-dimensional.
	Magnetic curves in a three-dimensional Sasakian manifold are studied\cite{Cabrerizo et al 2009}.
	Especially, magnetic curves in Sasakian space forms are classified\cite{Romaniuc et al 2015}.
	
	A contact metric manifold whose Reeb vector field is a Killing vector field is called K-contact.
	Sasakian manifold is K-contact and three-dimensional K-contact manifold is a Sasakian manifold.
	However a $2n+1(>3)$-dimensional K-contact manifold is not necessarily Sasakian manifold.
	And the following is known as necessary and sufficient condition of being K-contact for an odd dimensional Riemannian manifold.
	\begin{thm}
		[\cite{K.Yano and M.Kon} Chapter V Theorem 3.1]
		For an odd dimensional Riemannian manifold $(M,g)$, the necessary and sufficient condition of M being K-contact is (1) M admits a unit Killing vector field $\xi$, and (2) sectional curvature of arbitrary planes containing $\xi$ are equal to one.
		Moreover, the Reeb vector field of the Sasakian structure is $\xi$.
	\end{thm}
	\begin{proof}
		See Appendix \ref{appen sasaki manifold} Theorem \ref{appen thm k-contact}.
	\end{proof}
	Moreover, for a Riemannian manifold that admits a unit Killing vector field $\xi$, it is known that sectional curvature of arbitrary planes containing $\xi$ are non-negative\cite{Berestovskii Valerii Nikonrov 2008}(see appendix \ref{appen killing submersion})．
	Therefore we have the following.
	\begin{cor}
		\label{cor sectional curvature and sasaki magnetic field}
		Let $(M,g)$ be a three-dimensional orientable Riemannian manifold that admits a unit Killing vector field $\xi$.
		If sectional curvature of arbitrary planes containing $\xi$ are non-zero constant on M, then M is a Sasakian manifold and $\xi$ is the Reeb vector field of the Sasakian structure.
	\end{cor}
	
	\subsection{Riemannian Submersion}
	In this subsection we recall some definitions in the Riemannian submersion\cite{O'neill,Maria}.
	A differentiable map $\pi$ from a Riemannian manifold $(M,g)$ to a Riemannian manifold $(B,h)$ is a submersion if rank of $\pi_\ast$ is maximum.
	Vertical distribution is defined by $TM^v:=\ker\pi_\ast$ and horizontal distribution $TM^h$ is a tangent distribution that is orthogonal to the vertical distribution.
	Then a submersion $f$ is called Riemannian submersion if the restriction of $\pi_\ast$ to $TM^h$ is an isomorphism.
	An inverse image of $\pi$ defines a submanifold of M, called a fiber.
	And we assume that the dimension of fibers is larger than or equal to one.
	Let $\Gamma(D)$ be the linear space of all sections of a distribution $D$, then the orthogonal direct decomposition $\Gamma(TM)=\Gamma(TM^h)\oplus\Gamma(TM^v)$ holds with respect to the decomposition $T_pM=T_pM^h\oplus T_pM^v$ for all $p\in M$.
	And we write $\Gamma(TM)^c:=\{X\in\Gamma(TM);\ \pi_\ast X\in\Gamma(TB)\}$ and then an element of $\Gamma(TM)^b=\Gamma(TM^h)\cap\Gamma(TM)^c$ is called a basic vector field.
	For $X\in\Gamma(TM)$ and $X'\in\Gamma(TB)$, $X$ is $\pi-$related to $X'$ if $\pi_\ast X=X'$.
	If $X,Y\in\Gamma(TM)^b$ is $\pi-$related to $X',Y'$ then the relation of $\nabla$ and $\nabla'$ is given by
	\begin{align}
	\pi_\ast h\nabla_XY=\nabla'_{X'}Y',
	\end{align}
	i.e. $h\nabla_XY\in\Gamma(TM)^b$ is $\pi-$related to $\nabla'_{X'}Y'$.
	Here $\nabla$ and $\nabla'$ denotes Riemannian connections of $(M,g)$ and $(B,h)$ respectively.
	
	There exists two important tensors characterizing Riemannian submersion.
	Let $h$ and $v$ be projection from $\Gamma(TM)$ to $\Gamma(TM^h)$ and $\Gamma(TM^v)$, respectively.
	Fundamental tensors $A,T$ of Riemannian submersion are given by
	\begin{align}
	A_EF=h\nabla_{hE}vF+v\nabla_{hE}hF,\\
	T_EF=h\nabla_{vE}vF+v\nabla_{vE}hF,\\
	{\rm for}\ E,F\in\Gamma(TM).\nonumber
	\end{align}
	The tensor $A$ satisfies $A_XY=-A_YX=\frac{1}{2}v[X,Y]$ for $X,Y\in\Gamma(TM^h)$, and measures Frobenius integrability of $TM^h$, i.e. if $A=0$ then the horizontal distribution is integrable.
	And the tensor $T$ is, if one restricts it on $\Gamma(TM^v)$, a second fundamental form of fibers.
	The Fibers are totally geodesic if and only if $T=0$, and then we call such a Riemannian submersion a totally geodesic Riemannian submersion.
	
	A totally geodesic Riemannian submersion defines a fiber bundle\cite{Hermann}.
	For a totally geodesic Riemannian submersion, let $K$ be the sectional curvature of $M$, then the following
	\begin{align}
	K(X,V)=\frac{||A_XV||^2}{||X\wedge V||^2}\label{eq sectional curvature and A}\\
	{\rm for}\ X\in\Gamma(TM^h),\ V\in\Gamma(TM^v),
	\end{align}
	is known.
	This is a part of O'neill's formula.
	
	For a curve $c:I\rightarrow M$, we can decompose a tangent vector $X:=\dot{c}$ into $X(t)=E(t)+V(t),\ E\in\Gamma(TM^h),\ V\in\Gamma(TM^v)$.
	Since we have generally $h\nabla_XX=h\nabla_EE+2A_EV+T_VV$ then, for a totally geodesic Riemannian submersion,
	\begin{align}
	h\nabla_XY=h\nabla_EE+2A_EV,
	\label{eq horizontal decomposition of covariant derivative}
	\end{align}
	holds.

	\subsection{Killing Submersion}
	Let $(M,g)$ be a Riemannian manifold admits a unit Killing vector field $\xi$.
	Then sectional curvature of arbitrary planes containing $\xi$ are non-negative\cite{Berestovskii Valerii Nikonrov 2008}．
	In particular, if $M$ is an orientable and $\dim M=3$ then sectional curvature of an arbitrary plane containing $\xi$ depends on only an integral curve of $\xi$\cite{Espinar Oliveira 2013,Manzano 2014}．
	In detail, for arbitrary vector field $X$ we define a function $\tau$ by
	\begin{align}
	\nabla_X\xi=\tau X\wedge \xi.
	\end{align}
	We should note that the function $\tau$ is independent of the choice of X.
	Let $\pi:M\rightarrow B$ be a Riemannian submersion whose fibers are integral curves of $\xi$.
	This Riemannian submersion $\pi$ is called Killing submersion.
	Since the function $\tau$ is constant on each fibers, the scalar function $\tau$ defines the scalar function on $B$, called the bundle curvature of $\pi$.
	Furthermore, $\tau=K(X,\xi)$ holds, where $K$ denotes sectional curvature of $M$.
	We summarized proofs of these facts in Appendix \ref{appen killing submersion}. 
	
	If $\pi:M\rightarrow B$ is a Killing submersion then length of all fibers are same, finite or infinite(\cite{Manzano 2014} Lemma 2.3).
	If the basespace is simply connected then two Killing submersions are bundle isomorphic if length of fibers, bundle curvature and sectional curvature of the basespace are coincide up to isomorphism.
	Therefore, if the basespace is simply connected, a Killing submersion is characterized by $\tau$ and $\kappa$.
	We denote the total space by $M(\tau,\kappa)$.
	
	\section{Characterization of Sasakian Structure from a View Point of Magnetic Hopf Surfaces}\label{section Characterization of Sasaki structure from view point of magnetic Hopf surfaces}
	Let $\pi:M(\tau,\kappa)\rightarrow B$ be a Killing submersion of a unit Killing vector field $\xi$, where $\tau$ and $\kappa$ are the bundle curvature and sectional curvature of $B$ respectively.
	Let $c(t):I=(a,b)\rightarrow M$ be a normal magnetic curve for the Killing magnetic field $F_\xi$, where we take $b-a$ small enough.
	On the magnetic curve $c$, the angle $\theta_0$ between $\xi$ and $\dot{c}$ is constant since $\nabla_{\dot{c}}g(\dot{c},\xi)=g(\dot{c}\wedge\xi,\dot{c})+g(\dot{c},\nabla_{\dot{c}}\xi)=0$.
	If $\theta_0\ne0$, we call $\pi^{-1}\pi(c(t))$ the magnetic Hopf surface generated by $c$.
	Our main theorem is the following.
	\begin{thm}
		Let $\pi:M(\tau,\kappa)\rightarrow B$ be a Killing submersion of a unit Killing vector field $\xi$ and we assume that $M$ is an orientable simply connected complete three-dimensional Riemannian manifold with $\tau>0$.
		If an arbitrary short enough normal magnetic curve $c$ for the Killing magnetic field $F_\xi$ generates a magnetic Hopf surface of constant mean curvature, then M is a Sasakian manifold and $\xi$ is the Reeb vector field.
	\end{thm}
	\begin{proof}
		Let $c(t)$ be an enough short normal magnetic curve and put $\alpha(t)=\pi(c(t))$.
		We decompose $X:=\dot{c}(t)$ into horizontal and vertical direction, and extend suitably so that 
		\begin{align}
		X(t)=E(c(t))+V(c(t)),\\
		E\in\Gamma(TM)^b,V\in\Gamma(TM^v).
		\end{align}
		We assume $g(X,\xi)=\cos\theta_0,\ 0<\theta_0\le\pi/2$.
		From equation \eqref{eq horizontal decomposition of covariant derivative} we write the horizontal component of $\nabla_XX=X\wedge\xi$ for a magnetic curve as
		\begin{align}
		h\nabla_EE+2A_EV=X\wedge\xi.
		\end{align}
		
		We define the normal vector field $N\in\Gamma(TM)^b$ on $\Sigma:=\pi^{-1}\pi(c)$ such that $X\wedge\xi=\sin\theta_0N$.
		Since $g(A_EV,E)=-g(V,A_EE)=0$, $A_EV$ is proportional to $N$.
		Moreover, from equation\eqref{eq sectional curvature and A}, we have $||A_EV||^2=K(E,V)||E||^2||V||^2=\tau^2\sin^2\theta_0\cos^2\theta_0$ and then 
		\begin{align}
		A_EV=\epsilon\tau\sin\theta_0\cos\theta_0N,\ \epsilon=\pm1
		\end{align}
		holds.
		We have $g(A_EV,N)=g(\nabla_EV,N)=\tau||V|| g(E\wedge\xi,N)=\tau||V||\Omega(N,E,V)>0$, thus $\epsilon=1$ holds.
		Since $X\wedge\xi=\sin\theta_0N$ we obtain
		\begin{align}
		h\nabla_EE=\sin\theta_0(1-2\tau\cos\theta_0)N.
		\end{align}
		
		According to \cite{Espinar Oliveira 2013}, for a curve $\alpha$ whose geodesic curvature is $k_g$ on $B$, the mean curvature of the surface $\pi^{-1}(\alpha)$ of $M$ is given by $H=k_g/2$.
		Putting $T:=\dot{\alpha}$, the mean curvature of $\Sigma$ is $H=k_g/2=\sin\theta_0(1-2\tau\cos\theta_0)/2$ from $\pi_\ast h\nabla_EE=\nabla'_{T}T$, where $\nabla'$ is the Riemannian connection of $B$.
		From the assumption that the mean curvature of $\Sigma$ is constant, then $\tau$ is constant along $\alpha$.
		Therefore $\alpha(t)$ is a circle or a geodesic on $B$.
		
		We fix $\theta_0$ arbitrary, and we consider only normal magnetic curves $c(t)$ such that $g(\xi,\dot{c})=\cos\theta_0$.
		Since $\pi(c(t))$ is a magnetic curve on $B$ (or a geodesic on $B$), Lemma \ref{lem local completeness of circle on 2dim} (or geodesic completeness of $B$) leads to that arbitrary two points of $B$ is connected by piecewise projected magnetic curves.
		Then $\tau$ is constant on $B$.
		Therefore all bundle curvatures coincide and we obtain our claim by lemma \ref{cor sectional curvature and sasaki magnetic field}.
	\end{proof}

	\section{Conclusion}\label{section conclusion and disccution}
	We have inspected magnetic fields by motions of charged particles.
	We approach to this problem in simple mathematical settings and then we give a characterization of Sasakian manifold by using the geometric property of the surfaces generated by magnetic field lines and magnetic curves.
	For a three-dimensional Riemannian manifold $(M,g)$ that admits a unit Killing vector field $\xi$, we consider the Killing submersion $\pi:M\rightarrow B$ whose fibers are integral curves of $\xi$ with positive bundle curvature and we regard $\xi$ as magnetic vector field.
	We show that if an arbitrary normal magnetic curve $c$ for the Killing magnetic vector field $\xi$ generates a magnetic surface of constant mean curvature, then $M$ is a Sasakian manifold and $\xi$ is the Reeb vector field.

	\appendix
	\section{Killing Submersion}\label{appen killing submersion}
	In this appendix, we recall some facts in Killing submersion\cite{Berestovskii 2008,Espinar Oliveira 2013}．
	Through this section, $K(X,Y)$ denotes the sectional curvature of the plane spanned by $X$ and $Y$.
	\begin{prop}
		Let $(M,g)$ be a Riemannian manifold that admits a unit Killing vector field $\xi$.
		Then the sectional curvature of arbitrary planes containing $\xi$ are non-negative and the integral curve of $\xi$ is a geodesic.
	\end{prop}
	\begin{proof}
		For an arbitrary point $p\in M$ and an arbitrary unit vector $X\in T_pM$ orthogonal to $\xi$, let $x(s)$ be the geodesic starting from the initial point $p$ and the initial direction $X$, and the same symbol $X$ denotes $\dot{x}$.
		Since a Killing vector field is a Jacobi field for an arbitrary geodesic, $0=\frac{1}{2}\nabla_X\nabla_Xg(\xi,\xi)=g(\nabla_X\xi,\nabla_X\xi)+g(\nabla_X\nabla_X\xi,\xi)=g(\nabla_X\xi,\nabla_X\xi)-g(R(\xi,X)X,\xi)$ holds and then we have the first claim $K(X,\xi)=g(R(\xi,X)X,\xi)=||\nabla_X\xi||^2\geq0$.
		And the second claim follows from $0=(L_\xi g)(\xi,X)=\xi g(\xi,X)+g(\xi,[\xi,X])=g(\nabla_\xi\xi,X)-g(\xi,\nabla_X\xi)=g(\nabla_\xi\xi,X)-\frac{1}{2}Xg(\xi,\xi)=g(\nabla_\xi\xi,X)$ for an arbitrary $X$.
	\end{proof}
	
	$\nabla_X\xi$ is proportion to $X\wedge \xi$ since we have $g(\nabla_X\xi,X)=g(\nabla_X\xi,\xi)=0$ for an arbitrary vector field $X$.
	Then we have the following.
	\begin{prop}
		Let $(M,g)$ be a three-dimensional orientable Riemannian manifold that admits a unit Killing vector field $\xi$.
		For an arbitrary vector field $X$ we define the scalar function $\tau_X$ by
		\begin{align}
		\nabla_X\xi=\tau_X X\wedge \xi,
		\end{align}
		where $\nabla$ is the Riemannian connection with respect to $g$.
		Then $\tau=\tau_X$ is independent of $X$.
	\end{prop}
	\begin{proof}
		Let $\Omega$ be a volume form and let $\{X,Y,\xi\}$ be an orthogonal frame such that $\Omega(X,Y,\xi)=1$.
		It is enough to show $\tau_X=\tau_Y$.
		Since $\xi$ is a Killing vector field, then $0=g(\nabla_X\xi,Y)+g(\nabla_Y\xi,X)=\tau_X\Omega(X,\xi,Y)+\tau_Y\Omega(Y,\xi,X)=-\tau_X+\tau_Y$ holds.
	\end{proof}
	
	\begin{prop}
		In the above settings, $\tau_X^2=K(X,\xi)$ holds.
	\end{prop}
	\begin{proof}
		Since integral curves of $\xi$ are geodesics, the Riemannian submersion with respect to $\xi$ is a totally geodesic Riemannian submersion, then we have $K(X,\xi)=||A_X\xi||^2$ from O'Neill's formula\eqref{eq sectional curvature and A}.
		Since we have $g(A_X\xi,X)=g(A_X\xi,\xi)=0$ and $g(A_X\xi,Y)=g(\nabla_X\xi,Y)=\tau_X\Omega(X,\xi,Y)=-\tau_X$ then $A_X\xi=-\tau_X Y$ holds.
		Therefore $K(X,\xi)=\tau_X^2$ holds.
	\end{proof}
	From above two propositions we obtain the following.
	\begin{prop}
		Let $(M,g)$ be a three dimensional orientable Riemannian manifold that admits a unit Killing vector field $\xi$, and let $\pi:M\rightarrow B$ be a Riemannian submersion whose fibers are integral curves of $\xi$.
		Then $\tau$ defined above is a function on B. 
	\end{prop}
	
	\section{Sasakian Manifold}\label{appen sasaki manifold}
	In this section we summarize some definitions and facts about Sasakian manifold\cite{K.Yano and M.Kon}．
	Sasakian manifold is often considered as an analogy of K\"{a}hlerian manifold in an odd dimension.
	
	First we introduce almost contact and almost contact metric manifold.
	These are an analogy of almost complex and almost Hermite manifold respectively.
	On a $2n+1$ dimensional Riemannian manifold $M$, if there exists $(1,1)-$type tensor field $\phi$, 1-form $\eta$ and vector field $\xi$ such that
	\begin{align}
	\phi^2=-I+\eta\otimes\xi,\\
	\eta(\xi)=1,
	\end{align}
	then this triple $(\phi,\eta,\xi)$ is called an almost contact structure and $M$ is called an almost contact manifold.
	And then we also have $\phi(\xi)=0,\ \eta\circ\phi=0,\ {\rm rank}\phi=2n$.
	The vector field $\xi$ is called Reeb vector field of the almost contact structure $(\phi,\eta,\xi)$.
	Moreover, if there exists a Riemannian metric $g$ of $M$ satisfies
	\begin{align}
	g(\phi X,\phi Y)=g(X,Y)-\eta(X)\eta(Y),\label{eq metric condition}
	\end{align}
	for all vector field $X,Y$,	the 4-tuple $(\phi,\eta,\xi,g)$ is called an almost contact metric structure, and $M$ is called an almost contact metric manifold.
	In equation \eqref{eq metric condition}, setting $Y=\xi$, we obtain $\eta(X)=g(X,\xi)$.
	
	For an almost contact metric manifold $(M;\phi,\eta,\xi,g)$, we define the 2-form $F$ by $F(X,Y):=g(X,\phi Y)$.
	This is called a fundamental 2-form.
	If $F=d\eta$ holds, $(M;\phi,\eta,\xi,g)$ is called a contact metric manifold.
	This is an analogy of an almost K\"{a}hlerian manifold.
	On a contact metric manifold $(M;\phi,\eta,\xi,g)$, if Reeb vector field $\xi$ is a Killing vector field with respect to $g$ then $M$ is called K-contact.
	
	For an almost contact manifold $(M;\phi,\eta,\xi)$, we can consider a direct product manifold $M\times\mathbb{R}$.
	We can define an almost complex structure $J$ on  $M\times\mathbb{R}$ by
	\begin{align}
	J(X+f\partial_t)=\phi X-f\xi+\eta(X)\partial_t.
	\end{align}
	Here $X,\ f$ and $\partial_t$ are a vector field on $M$, a scalar function on $M\times\mathbb{R}$ and coordinate base of $\mathbb{R}$, respectively.
	If $J$ is integrable then an almost contact structure is called normal.
	Moreover, if an almost contact structure of a contact metric manifold is normal then the contact metric manifold is called Sasakian manifold.
	The following equivalent definition is often useful.
	
	It is known that, for an almost Hermite manifold $(M;J,g)$, $\nabla J=0$ is equivalent to being a K\"{a}hlerian manifold, where $\nabla$ is the Riemannian connection with respect to $g$.
	Similarly, for an almost contact metric manifold $(M;\phi,\eta,\xi,g)$, the condition
	\begin{align}
	(\nabla_X\phi)Y=g(X,Y)\xi-\eta(Y)X,\label{eq nabla phi}
	\end{align}
	is equivalent to being a Sasakian manifold.
	From this we have the following.
	\begin{prop}
		A Sasakian manifold $(M;\phi,\eta,\xi,g)$ is K-contact.
	\end{prop}
	\begin{proof}
		In equation \eqref{eq nabla phi} we put $Y=\xi$ then we have $-\phi\nabla_X\xi=\eta(X)\xi-X$ since $\phi(\xi)=0$.
		And we obtain $\nabla_X\xi=-\phi X$ by acting $\phi$ on both hand sides of it.
		Therefore $g(\nabla_X\xi,Y)+g(\nabla_Y\xi,X)=-g(\phi X,Y)-g(\phi Y,X)=0$ holds.
	\end{proof}
	The converse holds only in three-dimension\cite{Blair}.
	Moreover, the following condition is known as necessary and sufficient condition of being K-contact for an odd dimensional Riemannian manifold\cite{K.Yano and M.Kon}.
	\begin{thm}
		[\cite{K.Yano and M.Kon} Chapter V Theorem 3.1]
		For an odd dimensional Riemannian manifold $(M,g)$, the necessary and sufficient condition of M being K-contact is (1) M admits a unit Killing vector field $\xi$, and (2) sectional curvature of arbitrary planes containing $\xi$ are equal to one.
		Then $\xi$ is a Reeb vector field of the contact metric structure.
		\label{appen thm k-contact}
	\end{thm}
	\begin{proof}
		We assume $(g,\xi,\eta,\phi)$ is K-contact.
		Since $\xi$ is a Killing vector field, we have $R(X,\xi)Y=\nabla_X\nabla_Y\xi-\nabla_{\nabla_XY}\xi$.
		Then we obtain $R(X,\xi)\xi=-\nabla_{\nabla_X\xi}\xi=-\phi^2X$ by putting $Y=\xi$.
		Therefore, for an arbitrary unit vector $X$ orthogonal to $\xi$, $K(X,\xi)=g(R(X,\xi)\xi,X)=-g(\phi^2X,X)=g(\phi X,\phi X)=g(X,X)=1$ holds.
		
		Conversely we assume the conditions (1),(2).
		We define $\eta(X):=g(X,\xi),\ \phi(X):=-\nabla_X\xi$ and $K(X,\xi)$ denotes the sectional curvature of plane spanned by $X$ and $\xi$.
		For an arbitrary unit vector $X$ orthogonal to $\xi$, since $1=K(X,\xi)=g(R(X,\xi)\xi,X)=g(-\phi^2X,X)$, we have $\phi^2X=-X$.
		Therefore $\phi^2=-I+\eta\otimes\xi$ holds since $\phi(\xi)=0$.
		And from $g(\phi X,\phi X)=g(\nabla_X\xi,\nabla_X\xi)=||A_X\xi||^2=K(X,\xi)||X||^2=g(X,X)$, we obtain the equation \eqref{eq metric condition}.
		Finally, for arbitrary vector fields $X,Y$, we have $2d\eta(X,Y)=X\eta(Y)-Y\eta(X)-\eta([X,Y])=g(\nabla_X\xi,Y)-g(\nabla_Y\xi,X)=2g(\nabla_X\xi,Y)=2g(X,\phi Y)$.
	\end{proof}
	Finally we show a Sasakian magnetic vector field of a three-dimensional Sasakian manifold coincides the Reeb vector field.
	For a three-dimensional almost contact metric manifold, the same holds (\cite{Cabrerizo et al 2009}Prop5.1)．
	\begin{prop}
		\label{appen prop sasaki and killing magnetic field}
		In a three-dimensional Sasakian manifold $(M;\phi,\eta,\xi,g)$ with the volume form $\Omega$, $g(X,\phi Y)=\iota_\xi\Omega(X,Y)$ holds.
	\end{prop}
	\begin{proof}
		We have $\tau^2=1$ in $\nabla_X\xi=\tau X\wedge\xi$.
		We can take a suitable orientation such that the bundle curvature $\tau=-1$.
		Then $g(X,\phi Y)=-g(X,\nabla_Y\xi)=-g(X,\tau Y\wedge\xi)=\Omega(Y,\xi,X)=\iota_\xi\Omega(X,Y)$ holds.
	\end{proof}

\end{document}